\documentclass[12pt]{amsart}  

\usepackage{amsmath,amsthm,amssymb}
\usepackage[mathscr]{eucal}
\usepackage{enumerate}

\newtheorem{theorem}{Theorem}

\newtheorem{lemma}[theorem]{Lemma}
\newtheorem{corollary}[theorem]{Corollary}

\newcommand{\Z}{{\mathbb Z}}

\newcommand{\Q}{{\mathbb Q}}

\newcommand{\SL}{{\rm SL}_2 ({\mathbb Z})}

\title[Eisenstein series identities]{Eisenstein series identities based on partial fraction decomposition}
\author[M.~Hirose, N.~Sato, K.~Tasaka]{Minoru Hirose, Nobuo Sato, Koji Tasaka}
\thanks{This work is partially supported by Japan Society for the Promotion of Science, Grant-in-Aid for JSPS Fellows (No. 241440, No. 257323).}

\begin{document}

\maketitle

\begin{abstract}
From the theory of modular forms, there are exactly $[(k-2)/6]$ linear relations among the Eisenstein series $E_k$ and its products $E_{2i}E_{k-2i}\ (2\le i \le [k/4])$. 
We present explicit formulas among these modular forms based on the partial fraction decomposition, and use them to determining a basis of the space of modular forms of weight $k$ on $\SL$.
\end{abstract}

\section{Introduction}

The linear relations among the Eisenstein series $E_k(\tau)$ on $\SL$ and its products $E_{2i}(\tau)E_{k-2i}(\tau) \ (2\le i \le [k/4]) $ has been studied since the time of Liouville (see \cite{S}).
In the present paper, we study these relations by mainly using the partial fraction decomposition.

The main result of the current paper is as follows.
Throughout the paper, $\tau$ is a variable in the upper half-plane and $q=e^{2\pi \sqrt{-1}\tau}$.
We define the Eisenstein series $E_k(\tau)$ by 
\begin{equation}\label{eis}
 E_k(\tau) = \frac{2}{(k-1)!}\left( -\frac{B_k}{2k} +  \sum_{n>0} \sigma_{k-1}(n)q^n\right) \quad (k\ge2:{\rm even}),
\end{equation}
where $B_k$ is the $k$-th Bernoulli number and $\sigma_{k}(n)=\sum_{d|n}d^{k}$.
For convenience, we set $E_k(\tau)=0$ if $k\ge1$ is odd, and let 
\[ P_{r,s}(\tau) = E_r(\tau) E_s(\tau) +\delta_{r,2} \frac{E_s'(\tau)}{s} +\delta_{s,2} \frac{E_r'(\tau)}{r} ,\]
where the differential $'$ means $(2\pi \sqrt{-1})^{-1}d/d\tau$ and $\delta$ is the Kronecker delta.
We note that the function $P_{r,s}(\tau)$ is the product $E_r(\tau)E_s(\tau)$ for $r,s\ge4 \ ({\rm even})$, and becomes a modular form of weight $r+s$ on $\SL$ whenever $r,s\ge2$ are even.

\begin{theorem}\label{1} 
For positive integers $r,s,t\ge1$ with $k=r+s+t-1\ge3$, we have
\begin{equation} \label{eq2} 
\begin{aligned}
0=& \sum_{\substack{i+j=k\\i,j\ge1}}\binom{i-1}{t-1}\binom{j-1}{s-1}(-1)^{i+r} (P_{i,j}(\tau)-(-1)^jE_{i+j}(\tau)) \\
&+ \sum_{\substack{j+h=k\\j,h\ge1}}\binom{j-1}{r-1}\binom{h-1}{t-1}(-1)^{j+s} (P_{j,h}(\tau)-(-1)^hE_{j+h}(\tau)) \\
 & + \sum_{\substack{h+i=k\\h,i\ge1}}\binom{h-1}{s-1}\binom{i-1}{r-1}(-1)^{h+t} (P_{h,i}(\tau)-(-1)^iE_{h+i}(\tau)) .
\end{aligned}
\end{equation}
\end{theorem}

The identity \eqref{eq2} gives a lot of $\Q$-linear relations among $E_k(\tau)$ and $P_{2i,k-2i}(\tau) \ (1\le i \le [k/4]) $: for example, by taking $(r,s,t)=(1,2,2), (2,1,2)$ or $(2,2,1)$ in \eqref{eq2}, we have $5E_4(\tau)-P_{2,2}(\tau)=0$ (note that the right-hand side of \eqref{eq2} is obviously 0 if $k$ is odd).
In particular, we see that Theorem~\ref{1} gives all linear relations.
Denote by $M_k$ the space of modular forms of weight $k$ for $\SL$.
\begin{corollary}\label{2} 
For each even integer $k>2$, a basis of the space $M_k$ is given by the set
\[ \{ E_k(\tau)\} \cup \{  E_{2i}(\tau) E_{k-2i}(\tau) \mid i=[(k-2)/6]+2,[(k-2)/6]+3, \ldots, [k/4] \}. \]
\end{corollary}
We note that it is easy to check that the number of the above basis coincides with $\dim M_k$, because, for even $k>0$, we have $\dim M_k=[k/4]-[(k-2)/6]$.

In Section 2, we first show a certain identity which can be regarded as a kind of generalization of the partial fraction decomposition of $x^{-r}y^{-s}$. This identity plays an important role in the proof of Theorem~\ref{1}.
Finally, we prove Corollary~\ref{2}.


\section{Partial fraction decomposition and proofs}

We start with proving the following lemma:
\begin{lemma}\label{3}
Let $x,y,z$ be formal symbols with the relation $x+y+z=0$.
Then, for integers $r,s,t\ge1$, we have
\begin{align}
\label{eq3}0=& \sum_{\substack{i+j=k\\i,j\ge1}}\binom{i-1}{t-1}\binom{j-1}{s-1}(-1)^{i+r} x^{-i}y^{-j} + \sum_{\substack{j+h=k\\j,h\ge1}}\binom{j-1}{r-1}\binom{h-1}{t-1}(-1)^{j+s} y^{-j}z^{-h}\\
\notag & + \sum_{\substack{h+i=k\\h,i\ge1}}\binom{h-1}{s-1}\binom{i-1}{r-1}(-1)^{h+t} z^{-h}x^{-i},
\end{align}
where $k=r+s+t-1$. 
\end{lemma}
\begin{proof}
For integers $r,s,t\ge1$ we define a differential operator $D_{r,s,t}$ by 
\[ D_{r,s,t} = \left( \frac{\partial}{\partial x} -\frac{\partial}{\partial y}\right)^{r-1}\cdot \left( \frac{\partial}{\partial y} -\frac{\partial}{\partial z} \right)^{s-1}\cdot \left( \frac{\partial}{\partial z} - \frac{\partial}{\partial x} \right)^{t-1} .\]
One computes
\begin{align*}
D_{r,s,t}\left(\frac{1}{xy} \right)&= (-1)^{t-1} \left( \frac{\partial}{\partial x} -\frac{\partial}{\partial y}\right)^{r-1}\cdot  \left(\frac{\partial}{\partial y} \right)^{s-1} \cdot \left( \frac{\partial}{\partial x}\right)^{t-1}  \frac{1}{xy} \\
&=(-1)^{t-1} \sum_{i=0}^{r-1}(-1)^{i}\binom{r-1}{i} \left\{ \left( \frac{\partial}{\partial x}\right)^{t-1+r-1-i}  \frac{1}{x} \right\} \left\{ \left( \frac{\partial}{\partial y}\right)^{s-1+i}  \frac{1}{y} \right\} \\
&=(-1)^{r+s} \sum_{i=0}^{r-1} (-1)^i \frac{(r-1)!}{i!(r-1-i)!} \frac{ (t+r-i-2)! }{x^{t+r-1-i}} \frac{  (s+i-1)!}{y^{s+i}} \\
&= (-1)^{r+s}(r-1)!(s-1)!(t-1)! \sum_{i=0}^{r-1} \binom{t+r-i-2}{t-1}\binom{s+i-1}{i}\frac{(-1)^i}{x^{t+r-1-i}y^{i+s}}\\
(s+i\mapsto J)&=(r-1)!(s-1)!(t-1)! \sum_{\substack{I+J=k\\I,J\ge1}}\binom{I-1}{t-1}\binom{J-1}{s-1}(-1)^{I+r} \frac{1}{x^I y^J} .
\end{align*}
In the same manner, we can obtain 
\begin{align*}
D_{r,s,t}\big(y^{-1} z^{-1}\big)&=(r-1)!(s-1)!(t-1)!\sum_{\substack{j+h=k\\j,h\ge1}}\binom{j-1}{r-1}\binom{h-1}{t-1}(-1)^{j+s} y^{-j}z^{-h},\\
D_{r,s,t}\big(z^{-1} x^{-1}\big)&= (r-1)!(s-1)!(t-1)!\sum_{\substack{h+i=k\\h,i\ge1}}\binom{h-1}{s-1}\binom{i-1}{r-1}(-1)^{h+t} z^{-h}x^{-i}.
\end{align*}
Since $x^{-1}y^{-1}+y^{-1}z^{-1}+z^{-1}x^{-1}=0$, one has
\[0= D_{r,s,t} \big(x^{-1}y^{-1}+y^{-1}z^{-1}+z^{-1}x^{-1} \big) =(r-1)!(s-1)!(t-1)! \times (\mbox{R.H.S. of \eqref{eq3}}),\]
which completes the proof of \eqref{eq3}.\qed
\end{proof}

\noindent
{\it Remark}.
Putting $t=1$ in \eqref{eq3}, one easily obtains
\begin{align}
\notag0=& \sum_{\substack{i+j=r+s\\i,j\ge1}}\binom{j-1}{s-1}(-1)^{i+r} x^{-i}(-x-z)^{-j} + \sum_{\substack{j+h=r+s\\j,h\ge1}}\binom{j-1}{r-1}(-1)^{j+s} (-x-z)^{-j}z^{-h}\\
\notag & + \sum_{\substack{h+i=r+s\\h,i\ge1}}\binom{h-1}{s-1}\binom{i-1}{r-1}(-1)^{h+1} z^{-h}x^{-i}\\
\label{pfd}=&(-1)^s\Bigg[ \sum_{\substack{i+j=r+s\\i,j\ge1}}\left( \binom{i-1}{s-1} (x+z)^{-i}x^{-j} +\binom{i-1}{r-1} (x+z)^{-i}z^{-j}\right)-x^{-r} z^{-s}\Bigg],
\end{align}
which corresponds to the partial fraction decomposition of $x^{-r}z^{-s}$ (see \cite[eq.(19)]{GKZ}).

\noindent
{\bf Proof of Theorem \ref{1}}.
For $k\ge1$, using the Lipschitz formula, one can show that the Eisenstein series defined in \eqref{eis} coincides with the following limit: 
\[ E_k(\tau)=\lim_{M\rightarrow \infty} \sum_{m=-M}^{M} F_k(m,\tau),\]
where the function $F_k(m,\tau)$ is a holomorphic function on the upper half-plane defined by
\[ F_k(m,\tau)=\lim_{N\rightarrow \infty} \big( 2\pi \sqrt{-1} \big)^{-k}\sum_{\substack{n=-N\\(m,n)\neq(0,0)}}^{N}(m\tau+n)^{-k}.\]
We put the set $S_{M,N}=\{(m_1,m_2,n_1,n_2)\in \Z^4\mid -M\le m_1,m_2,m_1+m_2\le M , -N\le n_1,n_2,n_1+n_2\le N, (m_1,n_1)\neq(0,0), (m_2,n_2)\neq(0,0),(m_1+m_2,n_1+n_2)\neq(0,0)\}$, and define 
\[ G_{k_1,k_2,M,N}(\tau)=\big( 2\pi \sqrt{-1}\big)^{-k_1-k_2} \sum_{(m_1,m_2,n_1,n_2)\in S_{M,N}} (m_1\tau+n_1)^{-k_1} (m_2\tau+n_2)^{-k_2} .\] 
Then, letting $x=m_1\tau+n_1$ and $y=m_2\tau+n_2$ in \eqref{eq3} and summing up all elements in $S_{M,N}$, we have
\begin{equation*}
\begin{aligned}
0&= \sum_{\substack{i+j=k\\i,j\ge1}}\binom{i-1}{t-1}\binom{j-1}{s-1}(-1)^{i+r} \sum_{(m_1,m_2,n_1,n_2)\in S_{M,N}} (m_1\tau+n_1)^{-i}(m_2\tau+n_2)^{-j} \\
&+ \sum_{\substack{j+h=k\\j,h\ge1}}\binom{j-1}{r-1}\binom{h-1}{t-1}(-1)^{j+s} \sum_{(m_1,m_2,n_1,n_2)\in S_{M,N}}  (m_2\tau+n_2)^{-j}(-(m_1+m_2)\tau-(n_1+n_2))^{-h}\\
 & + \sum_{\substack{h+i=k\\h,i\ge1}}\binom{h-1}{s-1}\binom{i-1}{r-1}(-1)^{h+t}\sum_{(m_1,m_2,n_1,n_2)\in S_{M,N}}  (-(m_1+m_2)\tau-(n_1+n_2))^{-h}(m_1\tau+n_1)^{-i}\\
&=\big(2\pi \sqrt{-1} \big)^{k} \Bigg[ \sum_{\substack{i+j=k\\i,j\ge1}}\binom{i-1}{t-1}\binom{j-1}{s-1}(-1)^{i+r} G_{i,j,M,N}(\tau) \\
&+  \sum_{\substack{j+h=k\\j,h\ge1}}\binom{j-1}{r-1}\binom{h-1}{t-1}(-1)^{j+s} G_{j,h,M,N}(\tau) +\sum_{\substack{h+i=k\\h,i\ge1}}\binom{h-1}{s-1}\binom{i-1}{r-1}(-1)^{h+t} G_{h,i,M,N}(\tau) \Bigg],
\end{aligned}
\end{equation*}
because $(m_2,-m_1-m_2,n_2,-n_1-n_2)$ (resp. $(-m_1-m_2,m_1,-n_1-n_2,n_1)$) is in the set $S_{M,N}$ if and only if $(m_1,m_2,n_1,n_2)\in S_{M,N}$.
We note that when $r,s,t\ge3$, since for $k_1,k_2>2$ one has 
\begin{equation*}
\lim_{M\rightarrow \infty}\lim_{N\rightarrow \infty} G_{k_1,k_2,M,N}(\tau)=E_{k_1}(\tau)E_{k_2}(\tau)-(-1)^{k_2}E_{k_1+k_2}(\tau),
\end{equation*}
we can obtain
\begin{equation} \label{eq4} 
\begin{aligned}
0=& \sum_{\substack{i+j=k\\i,j\ge3}}\binom{i-1}{t-1}\binom{j-1}{s-1}(-1)^{i+r} (E_i(\tau)E_j(\tau)-(-1)^jE_{i+j}(\tau)) \\
&+ \sum_{\substack{j+h=k\\j,h\ge3}}\binom{j-1}{r-1}\binom{h-1}{t-1}(-1)^{j+s} (E_j(\tau)E_h(\tau)-(-1)^hE_{j+h}(\tau)) \\
 & + \sum_{\substack{h+i=k\\h,i\ge3}}\binom{h-1}{s-1}\binom{i-1}{r-1}(-1)^{h+t} (E_h(\tau)E_i(\tau)-(-1)^iE_{h+i}(\tau)) .
\end{aligned}
\end{equation}

To prove \eqref{eq2} for $r,s,t\ge1$, we begin by showing 
\begin{align}\label{test}
\lim_{M\rightarrow \infty} \lim_{N\rightarrow \infty} G_{k_1,k_2,M,N}(\tau) & =  E_{k_1}(\tau) E_{k_2}(\tau)-\frac{1}{2} \delta_{k_2,1} \frac{E_{k_1-1}'(\tau)}{k_1-1} - \frac{1}{2} \delta_{k_1,1} \frac{E_{k_2-1}'(\tau)}{k_2-1}-(-1)^{k_2}E_{k_1+k_2}(\tau),
\end{align}
for integers $k_1,k_2\ge1$ with $k_1+k_2\ge3$.
One computes
\begin{align*}
&\sum_{(m_1,m_2,n_1,n_2)\in S_{M,N}} \frac{1}{(m_1\tau+n_1)^{k_1} (m_2\tau+n_2)^{k_2}} \\
=& \sum_{-M\le m_1,m_2,m_1+m_2\le M} \sum_{\substack{-N\le n_1,n_2,n_1+n_2\le N\\ (m_1,n_1)\neq(0,0),(m_2,n_2)\neq(0,0)\\ (m_1+m_2,n_1+n_2)\neq(0,0)}} \frac{1}{(m_1\tau+n_1)^{k_1} (m_2\tau+n_2)^{k_2}}\\
=&  \sum_{-M\le m_1,m_2,m_1+m_2\le M}\Bigg(  \sum_{\substack{-N\le n_1,n_2,n_1+n_2\le N\\ (m_1,n_1)\neq(0,0)\\(m_2,n_2)\neq(0,0)}} -\sum_{\substack{-N\le n_1,n_2,n_1+n_2\le N\\ (m_1,n_1)\neq(0,0)\\(m_2,n_2)\neq(0,0)\\ (m_1+m_2,n_1+n_2)=(0,0)}}   \Bigg)\frac{1}{(m_1\tau+n_1)^{k_1} (m_2\tau+n_2)^{k_2}}\\
=& \sum_{-M\le m_1,m_2,m_1+m_2\le M} \Bigg( \sum_{\substack{-N\le n_1,n_2\le N\\ (m_1,n_1)\neq(0,0)\\(m_2,n_2)\neq(0,0)}} - \sum_{\substack{-N\le n_1,n_2\le N\\ n_1+n_2\not\in [ -N,N]\\ (m_1,n_1)\neq(0,0)\\(m_2,n_2)\neq(0,0)}} \Bigg)  \frac{1}{(m_1\tau+n_1)^{k_1}(m_2\tau+n_2)^{k_2}}\\
&-\sum_{\substack{-M\le m\le M\\ -N\le n\le N\\ (m,n)\neq(0,0)}} \frac{(-1)^{k_2}}{(m\tau+n)^{k_1+k_2}}.
\end{align*}
Since for integers $k_1,k_2\ge1$ with $k_1+k_2\ge3$ we have
\[ \lim_{N\rightarrow \infty} \sum_{\substack{-N\le n_1,n_2\le N\\ n_1+n_2\not\in [ -N,N]\\ (m_1,n_1)\neq(0,0)\\(m_2,n_2)\neq(0,0)}}  \frac{1}{(m_1\tau+n_1)^{k_1}(m_2\tau+n_2)^{k_2}} = 0,\]
we obtain
\begin{align}
\notag \lim_{N\rightarrow \infty} G_{k_1,k_2,M,N}(\tau)&=\sum_{-M\le m_1,m_2,m_1+m_2\le M} F_{k_1}(m_1,\tau) F_{k_2}(m_2,\tau) - (-1)^{k_1} \sum_{-M\le m \le M} F_{k_1+k_2} (m,\tau)\\
\label{eqeq1} &= \sum_{-M\le m_1,m_2\le M} F_{k_1}(m_1,\tau)F_{k_2}(m_2,\tau) -\sum_{\substack{-M\le m_1,m_2\le M\\ m_1+m_2\not\in [-M,M]}} F_{k_1}(m_1,\tau) F_{k_2,}(m_2,\tau) \\
\notag &- (-1)^{k_2} \sum_{-M\le m \le M} F_{k_1+k_2} (m,\tau).
 \end{align}
Noting that the function $F_k(m,\tau)+\frac{{\rm sgn}(m)}{2} \delta_{k,1}$ is a rapidly decreasing function of $m$, we have
\[\lim_{M\rightarrow \infty} \sum_{\substack{-M\le m_1,m_2\le M\\ m_1+m_2\not\in [-M,M]}} \Big(F_{k_1}(m_1,\tau)+\frac{{\rm sgn}(m_1)}{2} \delta_{k_1,1}\Big)\Big(F_{k_2}(m_2,\tau) +\frac{{\rm sgn}(m_2)}{2}\delta_{k_2,1}\Big)=0,\]
which gives for integers $k_1,k_2\ge1$ with $k_1+k_2\ge3$
\begin{align*}
&\lim_{M\rightarrow \infty} \sum_{\substack{-M\le m_1,m_2\le M\\ m_1+m_2\not\in [-M,M]}} F_{k_1}(m_1,\tau)F_{k_2}(m_2,\tau) \\
&= -\frac{1}{2} \delta_{k_2,1} \lim_{M\rightarrow \infty} \sum_{-M\le m\le M}m F_{k_1}(m,\tau)-\frac{1}{2} \delta_{k_1,1} \lim_{M\rightarrow \infty} \sum_{-M\le m\le M}m F_{k_2}(m,\tau)\\
&= \frac{1}{2} \delta_{k_2,1} \frac{E_{k_1-1}'(\tau)}{k_1-1} + \frac{1}{2} \delta_{k_1,1} \frac{E_{k_2-1}'(\tau)}{k_2-1}.
\end{align*}
Combining this with $\lim_{M\rightarrow\infty}$\eqref{eqeq1}, we obtain \eqref{test}.
For the proof of \eqref{eq2}, it suffices to check that
\begin{align*}
0=&\sum_{\substack{i+j=k\\i,j\ge1}}\binom{i-1}{t-1}\binom{j-1}{s-1}(-1)^{i+r}\Big(P_{i,j}(\tau)-(-1)^{j}E_{i+j}(\tau)- \lim_{M\rightarrow \infty} \lim_{N\rightarrow \infty} G_{i,j,M,N}(\tau)\Big) \\
&+\sum_{\substack{j+h=k\\j,h\ge1}}\binom{j-1}{r-1}\binom{h-1}{t-1}(-1)^{j+s} \Big(P_{j,h}(\tau)-(-1)^{h}E_{j+h}(\tau)- \lim_{M\rightarrow \infty} \lim_{N\rightarrow \infty} G_{j,h,M,N}(\tau)\Big)\\
&+\sum_{\substack{h+i=k\\h,i\ge1}}\binom{h-1}{s-1}\binom{i-1}{r-1}(-1)^{h+t} \Big(P_{h,i}(\tau)-(-1)^{i}E_{h+i}(\tau)- \lim_{M\rightarrow \infty} \lim_{N\rightarrow \infty} G_{h,i,M,N}(\tau)\Big)
\end{align*}
for even $k>0$.
From \eqref{test}, the right-hand side of the above can be written in the form
\begin{align*}
& \sum_{\substack{i+j=k\\i,j\ge1}}\binom{i-1}{t-1}\binom{j-1}{s-1}(-1)^{i+r} \frac{E_{k-2}'(\tau)}{k-2} \left( \delta_{j,2}+\delta_{i,2}+\frac{1}{2} \delta_{j,1}+\frac{1}{2} \delta_{i,1} \right) \\
&+ \sum_{\substack{j+h=k\\j,h\ge1}}\binom{j-1}{r-1}\binom{h-1}{t-1}(-1)^{j+s} \frac{E_{k-2}'(\tau)}{k-2} \left( \delta_{h,2}+\delta_{j,2}+\frac{1}{2} \delta_{h,1}+\frac{1}{2} \delta_{j,1} \right) \\
 & + \sum_{\substack{h+i=k\\h,i\ge1}}\binom{h-1}{s-1}\binom{i-1}{r-1}(-1)^{h+t}\frac{E_{k-2}'(\tau)}{k-2} \left( \delta_{i,2}+\delta_{h,2}+\frac{1}{2} \delta_{i,1}+\frac{1}{2} \delta_{h,1} \right).
\end{align*}
It can be easily shown that the above is 0: for example, when $t=2$ and $r,s\ge3$, the above equation is $0$ because of 
\begin{align*}
\frac{E_{k-2}'(\tau)}{k-2}\left\{\binom{k-3}{s-1}(-1)^r+\binom{k-3}{r-1}(-1)^{s}\right\}=0,
\end{align*}
and also, when $t=1$ and $r,s\ge3$, the above equation can be reduced to 
\begin{align*}
\frac{E_{k-2}'(\tau)}{k-2}\left\{ \binom{k-3}{s-1}(-1)^r-\frac12\binom{k-2}{s-1}(-1)^{r} + \binom{k-3}{r-1} (-1)^{s}-\frac12 \binom{k-2}{r-1}(-1)^{s} \right\},
\end{align*}
which is 0.
We complete the proof.
\qed

\

\noindent
{\bf Proof of Corollary \ref{2}}.
Let $d_k=\dim M_k=[k/4]-[(k-2)/6]$.
The space $M_k$ is generated by the set $\{ E_k,E_{2i}E_{k-2i}\mid 2\le i \le  [k/4]\}$ (see \cite{F,GKZ}).
Therefore, it is enough to show that each $E_{2i}E_{k-2i}\ (2\le i \le  [k/4]-d_{k}+1)$ can be expressed as sums of $E_{k}$ and $E_{2j}E_{k-2j} \ (i+1\le j \le  [k/4])$.
For each $i \ (2\le i \le [k/4]-d_{k}+1)$, we take $(r,s,t)=(2i-1,a,b)$ with odd integers $a,b\ge2i-1$ such that $a+b=k-2i+2$ (for example, we can choose $(2i-1,2i-1,k-4i+3)$ for any $i \in \{2,3,\ldots,[(k-2)/6]+1\}$).
Substituting this into \eqref{eq4}, we find that the coefficients of $E_{2p}E_{k-2p}\ (2\le p\le i-1)$ in \eqref{eq4} are 0 and the coefficient of $E_{2i}E_{k-2i}$ in \eqref{eq4} is a negative integer since $(2i-1,a,b)\equiv (1,1,1) \pmod{2}$, which gives our demanded relation.
\qed

\

\noindent
 {\it Remark}. 
(i) Corollary~\ref{2} is similar to the result obtained by Fukuhara~\cite[Theorem~1,1]{F}, but not the same.\\
(ii) Specializing $t=1$ in \eqref{eq2}, for $r+s\ge4$ (even) we obtain
\begin{equation}\label{eq1}
0 =  \sum_{\substack{i+j=r+s\\i,j\ge1}} \left(\binom{i-1}{r-1}+\binom{i-1}{s-1} \right) (P_{i,j}(\tau) -  E_{i+j}(\tau)) -P_{r,s}(\tau)+(-1)^sE_{r+s}(\tau) . 
\end{equation}
The identity \eqref{eq1} was already shown by Popa \cite[(A.3)]{P}, which was first indicated by Zagier \cite[\S 8]{Z} (Zagier also gave a direction of the proof of \eqref{eq1} based on partial fraction decomposition \eqref{pfd} (see \cite{Z1})).
Popa pointed out that the right-hand side of \eqref{eq1} except $ E_k(\tau)$ is orthogonal to all Hecke eigen cusp forms of weight $k$ on $\SL$ under the Petersson product.
This means the right-hand side of \eqref{eq1} except $E_k(\tau)$ has to be $c E_k(\tau)$ for some constant $c \in \Q$, and he determined the constant $c=\binom{k}{r} - (-1)^{s}$ using an interesting Bernoulli numbers identities.
(The constant $c$ is exactly equal to our coefficient of $E_k$ since $\sum_{r+s=k} \sum_{i+j=k}(\binom{i-1}{r-1}+\binom{i-1}{s-1} ) x^{r-1}y^{s-1}=((x+y)^k-x^k-y^k)/xy = \sum_{r+s=k} \binom{k}{r} x^{r-1}y^{s-1}$, where the variables $r,s,i,j$ run over $\Z_{>0}$.)





\

\small{
\noindent
\address{Graduate~School~of~Mathematics, Kitashirakawa Oiwake-cho, Sakyo-ku, Kyoto, 606-8502, Japan}\\
Minoru Hirose\\
\email{hirose@math.kyoto-u.ac.jp }\\
\address{Graduate~School~of~Mathematics, Kitashirakawa Oiwake-cho, Sakyo-ku, Kyoto, 606-8502, Japan}\\
Nobuo Sato\\
\email{sato@math.kyoto-u.ac.jp }\\
\address{Graduate~School~of~Mathematics, 744, Motooka, Nishi-ku, Fukuoka, 819-0395, Japan}\\
Koji Tasaka\\
\email{k-tasaka@math.kyushu-u.ac.jp }
}
\noindent
         
         
              
              

\end{document}